\documentclass[12pt]{amsart}
\usepackage{amsfonts,amsthm,amsmath,amssymb,verbatim}
\usepackage{graphicx}
\usepackage[active]{srcltx}
\usepackage[all,cmtip]{xy}
\usepackage{tikz}
\usepackage{xypic}
\usepackage{cancel}

\newtheorem{thm}{Theorem}[section]

\newtheorem{lemma}[thm]{Lemma}

\newtheorem{cor}[thm]{Corollary}

\theoremstyle{remark}

\newtheorem{example}[thm]{Example}
\newtheorem{remark}[thm]{Remark}
\newtheorem{defin}{Definition}



\def\C{\mathbb{C}}
\def\R{\mathbb{R}}

\def\P{\mathbb{P}}

\def\d{\partial}

\def\a{\alpha}

\def\l{\lambda}
\def\D{\Delta}
\def\G{\Gamma}

\def\M{{\rm mitosis}}
\def\mM{{\rm{\underline{mitosis}}}}
\def\MA{M^A}
\def\MC{M^C}
\def\env{{\rm env}}

\title{Simple geometric mitosis}
\author{Valentina Kiritchenko}
\email{vkiritch@hse.ru}
\thanks{The study has been partially funded within the framework of the HSE University Basic Research Program.}

\address{Laboratory of Algebraic Geometry and Faculty of Mathematics\\
National Research University Higher School of Economics\\
Usacheva str. 6, 119048 Moscow, Russia}
\address{Institute for Information Transmission Problems, Moscow, Russia}
\date{}

\begin{document}
\begin{abstract}
We construct simple geometric operations on faces of the Cayley sum of two polytopes.
These operations can be thought of as convex geometric counterparts of divided difference operators in Schubert calculus.
We show that these operations give a uniform construction of Knutson--Miller mitosis (in type $A$) and (simplified) Fujita mitosis (in type $C$) on Kogan faces of Gelfand--Zetlin polytopes.
\end{abstract}

\maketitle

\section{Introduction}
Mitosis as a mathematical notion was introduced by Knutson and Miller to generate 
combinatorial objects (reduced pipe dreams or rc-graphs) from a single object \cite{KnM,M}.
Originally, pipe dreams were used in Fomin--Kirillov theorem to enumerate the coefficients of Schubert polynomials in type $A$.
Recall that Schubert polynomials are defined inductively by applying divided difference operators to a single Schubert polynomial. 
Hence, mitosis operations can be thought of as combinatorial counterparts of divided difference operators from Schubert calculus.
A convex geometric version of mitosis was defined in \cite{K16}: pipe dreams were replaced by faces of polytopes from representation theory (such as Nakashima--Zelevinsky polytopes) and divided difference operators were replaced by Demazure operators.
In the present paper, we introduce a different version of geometric mitosis (simple geometric mitosis).
It is more general and can be applied to classical Gelfand--Zetlin polytopes in all types.
The starting point for this paper was the desire to understand the results of \cite{F} from a geometric rather than representation theoretic viewpoint. 
  
Simple geometric mitosis is defined on a convex polytope $\D$ such that $\D$ can be represented as the Cayley sum of two of its facets.
Let $v$ be a vertex of $\Delta$.
Faces of $\Delta$ that contain $v$ play the role of pipe dreams.
Mitosis operation acts only on faces that contain $v$ and assigns to a face a collection of faces of dimension one greater.
In applications to Schubert calculus, there are several different mitosis operations associated with $\Delta$.
This is because Gelfand--Zetlin polytopes (and other polytopes from representation theory) admit several decompositions into Cayley sums of facets.
In this case, mitosis operations generate faces (that contain $v$) from the single vertex $v$.
The faces generated this way can be used to represent Schubert cycles as sums of faces.

The paper is organized as follows. 
In Section \ref{s.mitosis}, we give a definition of simple geometric mitosis.
In Section \ref{s.comb}, we recall the combinatorial definition of Knutson--Miller mitosis on pipe dreams in type $A$ and Fujita mitosis on skew pipe dreams in type $C$.
In Section \ref{s.Schubert}, we relate the constructions of the previous two sections using Gelfand--Zetlin polytopes in types $A$ and $C$.
In Section \ref{s.Schubert}, we outline applications to Schubert calculus.

\section{Main construction} \label{s.mitosis}
In this section, we give an elementary definition of simple geometric mitosis.
In Section \ref{s.ppo}, we explain the meaning of this definition in the setting of intersection theory.

Let $P$, $Q\subset \R^d$ be two convex polytopes of full dimension.
Denote by $\D:=\D(P,Q)\subset\R^d\times\R$ their Cayley sum $P*Q$, that is, the convex hull of
$$(P\times 0)\cup (Q\times 1).$$
In what follows, we identify $P$ and $Q$ with facets $P\times 0$ and $Q\times 1$ of $\D$.
The {\em mitosis operation} acts only on those faces of $\D$ that are contained in $P$ (``horizontal faces''), and produces faces of $\D$ that are not contained in $P$.

\begin{defin}\label{d.adm}
The face $F\subset P$ is called {\em admissible}, if there exists a unique face $\exp(F)\ne F$ with the property
$$F=\exp(F)\cap P.$$
\end{defin}
It is easy to check that if $F$ is admissible then $\dim\exp(F)=\dim F+1$.
The face $\exp(F)$ can be thought of as an expansion of $F$ from $P$ to $\Delta$.

\begin{example}
Let $d=2$, and let $P, Q\subset \R^2$ be two triangles obtained from a unit square by cutting along a diagonal.
More precisely, take $P=\{(x_1,x_2)\in\R^2 \ | \ x_1,x_2 \le 1, \ x_1+x_2\ge 1\}$
and $Q=\{(x_1,x_2)\in\R^2 \ | \ x_1,x_2 \ge 0, \ x_1+x_2\le 1\}$.
Then $\Delta=\{(x_1,x_2,x_3)\in\R^3 \ | \ x_1,x_2,x_3\le 1, \ 1\le x_1+x_2+x_3 \le 2\}$.
The face $P\subset\D$ and edges of $P$ are  admissible faces, while vertices of $P$ are not.
\end{example}

Let $v\in P$ be a vertex.
In the definition below, we consider only those faces of $P$ and $\D$ that contain $v$.
We call them {\em $v$-faces}.
In particular, the {\em mitosis operation} $\M^v(\cdot)$ will depend on the choice of $v$.
Let $F\subset P$ be an admissible $v$-face of dimension $\ell$.
We will define $\M^v(F)$ geometrically as a set of $v$-faces  $\{E_1,\ldots,E_k\}$ of dimension $\ell+1$.

\begin{defin}\label{d.mitosis}
A $v$-face $E_i\subset \Delta$ belongs to $\M^v(F)$ if $E_i$ satisfies the following two conditions:
\begin{enumerate}
\item $E_i\not\subset P$, $Q$;
\item $E_i\cap Q\subset \exp(F)\cap Q$.
\end{enumerate}
Faces in $\M^v(F)$ will be called {\em offsprings} of $F$.
\end{defin}
Informally, the first condition means that $E_i$ is not a ``horizontal'' face, that is, $E_i$ intersects both $P$ and $Q$.
In particular, it is not possible to apply the same mitosis operation twice because none of the faces $E_i\in\M^v(F)$ lies in $P$.
The second condition tells us that the face $E_i$ is not in general position with respect to $Q$ (unless $E_i=\exp(F)$).

\begin{example} \label{e.FFLV} Let $\Delta\subset\R^3$ be a Feigin--Fourier--Littelmann--Vinberg (FFLV) polytope given by inequalities $0\le x_1, x_3\le 1$, $0\le x_2$,  $x_1+x_2+x_3\le 2$.
Consider faces:
$$P_1=\{x_3=0\}, \quad Q_1=\{x_3=1\}, \quad P_2=\{x_1=1\}, \quad Q_2=\{x_1=0\}.$$
Let $v$ be the vertex $(1,1,0)$.
Clearly, there are two ways to decompose $\Delta$ as the Cayley sum of two polygons: $\Delta=\Delta(P_1,Q_1)$ and $\Delta=\Delta(P_2,Q_2)$.
Hence, there are two different mitosis operations $\M^v_1$ and $\M^v_2$ associated with these decompositions.

For instance, $\M^v_1(F)$ for the edge $F=\{x_3=0, \ x_1+x_2+x_3=2\}$ consists of two offsprings, namely, $\exp(F)=\{x_1+x_2+x_3=2\}$ and
$P_2$.
However, $\M^v_2(F)$ consists of a single offspring for all admissible $v$-faces $F\subset P_2$.
\end{example}
\begin{remark}
In what follows, we sometimes apply the mitosis operation to a set $S$ of faces of the same dimension.
By $\M^v(S)$ we mean
$\bigcup_{F\in S}\M^v(F).$
\end{remark}
\section{Combinatorial mitosis}\label{s.comb}
In this section, we recall two combinatorial rules: Knutson--Miller mitosis on pipe dreams (type $A$) and Fujita mitosis on skew pipe dreams (type $C$).
Both rules can be defined uniformly using the same combinatorial operation called {\em two-row mitosis}.
A similar operation is implicitly used in the original definition of Knutson--Miller mitosis \cite{KnM, M}.
We first define {\em two-row mitosis} explicitly, and then define mitosis operations in types $A$ and $C$ by reducing them to suitable two-row mitosis.
Note that the term {\em row} here does not necessarily mean a horizontal collection of items as in the original definition of Knutson--Miller mitosis.

\subsection{Two-row mitosis}
Let $A$ and $B$ be two finite collections of squares such that the number of squares in $A$ is greater by one than the number of squares in $B$.
Denote the number of squares in $B$ by $\ell$.
We will label the squares in $A$ by $a1$, $a2$,\ldots, $a(\ell+1)$, and the squares in $B$ by $b1$, $b2$,\ldots, $b\ell$ (like in chess notation).
Symbolically, we may represent $A$ and $B$ by $(\ell+1)$-row and $\ell$-row of squares, respectively.
In real life examples, we might need to arrange squares of $A$ and $B$ in more intricate ways.
Namely, we will identify squares in $A$ and $B$ with specific cells in various tables.

To get a {\em basic pipe dream} we fill some squares in $A$ and $B$ with $+$.
The other squares remain empty.
By {\em size} of a basic pipe dream $D$ we mean the total number of $+$ in the squares of $D$.
{\em Two-row mitosis} $M$ is an operation on basic pipe dreams that sends a basic pipe dream $D$ of size $s$ to a (possibly empty) set  $M(D)$ of basic pipe dreams of size $(s-1)$.

To construct the set $M(D)$ we use the following rules:
\begin{enumerate}
\item If square $a1$ is empty, then $M(D)=\varnothing$.
\item If square $a1$ contains $+$, then find the maximal index $r_D$ such that the squares $a1$,\ldots, $ar_D$ are all filled with $+$.
\item If square $a1$ contains $+$, define the set $\mathcal J(D)$ of indices:
$$\mathcal J(D):=\{j\le r_D \ |\ \mbox{ square } aj \mbox{ contains } +, \mbox{ square } bj \mbox{ is empty or } j=\ell+1\}.$$
\item For every $p\in \mathcal J(D)$, construct the offspring $D_p$ as follows.
First, erase $+$ in square $ap$.
Then move $+$ from square $aj$ down to square $bj$ for all $j\in \mathcal J(D)$ such that $j<p$.
\end{enumerate}

\begin{example}{(cf. \cite[Example 7]{M})}
\label{e.2mitosis}
Let $\ell=5$.
Figure \ref{f.2row} shows a pipe dream $D$ of size $6$ and three pipe dreams of size $5$ that form the set $M(D)$.
In this case, $\mathcal J(D)=\{1,2,4\}$, and
$r_D=4$.
\begin{figure}[h]
		\centering
		\begin{tikzpicture}[x=1em,y=-1em]
			\draw (0,0) rectangle +(1,1);
			\draw (1,0) rectangle +(1,1);
			\draw (2,0) rectangle +(1,1);
			\draw (3,0) rectangle +(1,1);
            \draw (4,0) rectangle +(1,1);
            \draw (5,0) rectangle +(1,1);
			\draw (0,1) rectangle +(1,1);
			\draw (1,1) rectangle +(1,1);
			\draw (2,1) rectangle +(1,1);
            \draw (3,1) rectangle +(1,1);
            \draw (4,1) rectangle +(1,1);
			\draw (0.5,0.5) node{$+$};
			\draw (1.5,0.5) node{$+$};
			\draw (2.5,0.5) node{$+$};
            \draw (3.5,0.5) node{$+$};
            \draw (2.5,1.5) node{$+$};
		\end{tikzpicture}
$$\downarrow$$
\begin{tikzpicture}[x=1em,y=-1em]
			\draw (0,0) rectangle +(1,1);
			\draw (1,0) rectangle +(1,1);
			\draw (2,0) rectangle +(1,1);
			\draw (3,0) rectangle +(1,1);
            \draw (4,0) rectangle +(1,1);
            \draw (5,0) rectangle +(1,1);
			\draw (0,1) rectangle +(1,1);
			\draw (1,1) rectangle +(1,1);
			\draw (2,1) rectangle +(1,1);
            \draw (3,1) rectangle +(1,1);
            \draw (4,1) rectangle +(1,1);
			\draw (1.5,0.5) node{$+$};
			\draw (2.5,0.5) node{$+$};
            \draw (3.5,0.5) node{$+$};
            \draw (2.5,1.5) node{$+$};
		\end{tikzpicture}
\begin{tikzpicture}[x=1em,y=-1em]
			\draw (0,0) rectangle +(1,1);
			\draw (1,0) rectangle +(1,1);
			\draw (2,0) rectangle +(1,1);
			\draw (3,0) rectangle +(1,1);
            \draw (4,0) rectangle +(1,1);
            \draw (5,0) rectangle +(1,1);
			\draw (0,1) rectangle +(1,1);
			\draw (1,1) rectangle +(1,1);
			\draw (2,1) rectangle +(1,1);
            \draw (3,1) rectangle +(1,1);
            \draw (4,1) rectangle +(1,1);
			\draw (2.5,0.5) node{$+$};
            \draw (3.5,0.5) node{$+$};
            \draw (0.5,1.5) node{$+$};
            \draw (2.5,1.5) node{$+$};
		\end{tikzpicture}
\begin{tikzpicture}[x=1em,y=-1em]
			\draw (0,0) rectangle +(1,1);
			\draw (1,0) rectangle +(1,1);
			\draw (2,0) rectangle +(1,1);
			\draw (3,0) rectangle +(1,1);
            \draw (4,0) rectangle +(1,1);
            \draw (5,0) rectangle +(1,1);
			\draw (0,1) rectangle +(1,1);
			\draw (1,1) rectangle +(1,1);
			\draw (2,1) rectangle +(1,1);
            \draw (3,1) rectangle +(1,1);
            \draw (4,1) rectangle +(1,1);
			\draw (2.5,0.5) node{$+$};
            \draw (0.5,1.5) node{$+$};
            \draw (1.5,1.5) node{$+$};
            \draw (2.5,1.5) node{$+$};
		\end{tikzpicture}
		\caption{}
		\label{f.2row}
	\end{figure}
\end{example}

\subsection{Mitosis in type A}
Recall that a pipe dream in type $A_n$ can be defined as an $n\times n$ table whose cells are either filled with $+$ or empty.
Furthermore, $+$ is not allowed in cell $(i,j)$ if the cell is below the main antidiagonal of the table (that is, $i+j>n+1$).
An example of pipe dream for $n=6$ is given on Figure \ref{f.pipe_dream} (left).
Pipe dreams have an elegant interpretation in terms of networks of pipes, and are related to permutations from $S_{n+1}$ (see \cite{M} for details).
\begin{figure}[h]
		\centering
		\begin{tikzpicture}[x=1em,y=-1em]
			\draw (0,0) rectangle +(1,1);
			\draw (1,0) rectangle +(1,1);
			\draw (2,0) rectangle +(1,1);
			\draw (3,0) rectangle +(1,1);
            \draw (4,0) rectangle +(1,1);
            \draw (5,0) rectangle +(1,1);
			\draw (0,1) rectangle +(1,1);
			\draw (1,1) rectangle +(1,1);
			\draw (2,1) rectangle +(1,1);
            \draw (3,1) rectangle +(1,1);
            \draw (4,1) rectangle +(1,1);
			\draw (0,2) rectangle +(1,1);
			\draw (1,2) rectangle +(1,1);
            \draw (2,2) rectangle +(1,1);
            \draw (3,2) rectangle +(1,1);
            \draw (0,3) rectangle +(1,1);
			\draw (1,3) rectangle +(1,1);
            \draw (2,3) rectangle +(1,1);
            \draw (0,4) rectangle +(1,1);
            \draw (1,4) rectangle +(1,1);
            \draw (0,5) rectangle +(1,1);
            \draw (0.5,0.5) node{$+$};
			\draw (1.5,0.5) node{$+$};
			\draw (2.5,0.5) node{$+$};
            \draw (3.5,0.5) node{$+$};
            \draw (2.5,1.5) node{$+$};
            \draw (0.5,2.5) node{$+$};
            \draw (0.5,3.5) node{$+$};
            \draw (0.5,4.5) node{$+$};
		\end{tikzpicture} \qquad
\begin{tikzpicture}[x=1em,y=-1em]
			\draw (0,0) rectangle +(1,1);
			\draw (1,0) rectangle +(1,1);
			\draw (2,0) rectangle +(1,1);
			\draw (3,0) rectangle +(1,1);
            \draw (4,0) rectangle +(1,1);
            \draw (5,0) rectangle +(1,1);
			\draw (6,0) rectangle +(1,1);
			\draw (1,1) rectangle +(1,1);
            \draw (2,1) rectangle +(1,1);
            \draw (3,1) rectangle +(1,1);
			\draw (4,1) rectangle +(1,1);
			\draw (5,1) rectangle +(1,1);
            \draw (2,2) rectangle +(1,1);
            \draw (3,2) rectangle +(1,1);
            \draw (4,2) rectangle +(1,1);
			\draw (3,3) rectangle +(1,1);
            \draw (3.5,0.5) node{$+$};
			\draw (3.5,1.5) node{$+$};
			\draw (3.5,2.5) node{$+$};
            \draw (3.5,3.5) node{$+$};
			\draw (4.5,2.5) node{$+$};
            \draw (2.5,1.5) node{$+$};
            \draw (2.5,2.5) node{$+$};
            \draw (2.5,0.5) node{$+$};
            \draw (0.5,0.5) node{$+$};
		\end{tikzpicture}
		\caption{Pipe dreams in type $A_6$ (left) and $C_4$ (right).}
		\label{f.pipe_dream}
	\end{figure}
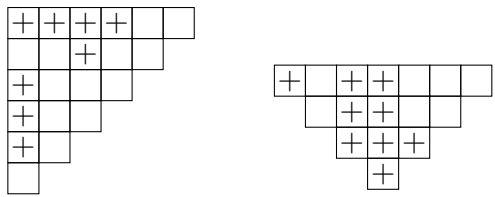

There are $n$ different mitosis operations $\MA_1$,\ldots, $\MA_n$ on pipe dreams in type $A_n$.
Informally, mitosis operation $\MA_i$ can be defined as the two-row mitosis applied to rows $i$ and $i+1$ of a pipe dream (the other rows are not affected by $\MA_i$).
Note that mitosis operation $\MA_n$ is also well-defined though row $n+1$ is an empty set of boxes.
This is because row $n$ might have at most one $+$, that is, no $+$ will have to be moved down according to the mitosis rules.

We now define mitosis operation $\MA_i$ more formally.
Let $D$ be a pipe dream in type $A_n$.
Put $\ell=n-i$, and label cells $(i,1)$,\ldots, $(i,n-i+1)$ by $a1$,\ldots, $a(\ell+1)$, respectively.
Label cells $(i+1,1)$,\ldots, $(i+1,n-i)$ by $b1$,\ldots, $b\ell$, respectively.
Extract the basic pipe dream $D^i$ from $D$ by setting $A^i=(a1,\ldots, a(\ell+1))$ and $B^i=(b1,\ldots, b\ell)$.
Apply two-row mitosis to $D^i$.
Complete every resulting offspring $D^i_p$ to a pipe dream $D_p$ by replacing $D^i$ with $D^i_p$ in $D$.
Define $\MA_i(D)$ as the set $\{D_p \ | \ p\in \mathcal J(D^i)\}$.

\begin{example} Let $D$ be a pipe dream depicted on Figure \ref{f.pipe_dream} (left).
For $i=1$, the basic pipe dream $D^1$ coincides with the one on Figure \ref{f.2row} (top) from Example \ref{e.2mitosis}.
Hence, $\MA_1(D)$ consists of the pipe dreams shown on Figure \ref{f.mitosisA}.
\begin{figure}[h]
		\centering
		\begin{tikzpicture}[x=1em,y=-1em]
			\draw (0,0) rectangle +(1,1);
			\draw (1,0) rectangle +(1,1);
			\draw (2,0) rectangle +(1,1);
			\draw (3,0) rectangle +(1,1);
            \draw (4,0) rectangle +(1,1);
            \draw (5,0) rectangle +(1,1);
			\draw (0,1) rectangle +(1,1);
			\draw (1,1) rectangle +(1,1);
			\draw (2,1) rectangle +(1,1);
            \draw (3,1) rectangle +(1,1);
            \draw (4,1) rectangle +(1,1);
			\draw (0,2) rectangle +(1,1);
			\draw (1,2) rectangle +(1,1);
            \draw (2,2) rectangle +(1,1);
            \draw (3,2) rectangle +(1,1);
            \draw (0,3) rectangle +(1,1);
			\draw (1,3) rectangle +(1,1);
            \draw (2,3) rectangle +(1,1);
            \draw (0,4) rectangle +(1,1);
            \draw (1,4) rectangle +(1,1);
            \draw (0,5) rectangle +(1,1);
            \draw (1.5,0.5) node{$+$};
			\draw (2.5,0.5) node{$+$};
            \draw (3.5,0.5) node{$+$};
            \draw (2.5,1.5) node{$+$};
            \draw (0.5,2.5) node{$+$};
            \draw (0.5,3.5) node{$+$};
            \draw (0.5,4.5) node{$+$};
		\end{tikzpicture} \quad
\begin{tikzpicture}[x=1em,y=-1em]
			\draw (0,0) rectangle +(1,1);
			\draw (1,0) rectangle +(1,1);
			\draw (2,0) rectangle +(1,1);
			\draw (3,0) rectangle +(1,1);
            \draw (4,0) rectangle +(1,1);
            \draw (5,0) rectangle +(1,1);
			\draw (0,1) rectangle +(1,1);
			\draw (1,1) rectangle +(1,1);
			\draw (2,1) rectangle +(1,1);
            \draw (3,1) rectangle +(1,1);
            \draw (4,1) rectangle +(1,1);
			\draw (0,2) rectangle +(1,1);
			\draw (1,2) rectangle +(1,1);
            \draw (2,2) rectangle +(1,1);
            \draw (3,2) rectangle +(1,1);
            \draw (0,3) rectangle +(1,1);
			\draw (1,3) rectangle +(1,1);
            \draw (2,3) rectangle +(1,1);
            \draw (0,4) rectangle +(1,1);
            \draw (1,4) rectangle +(1,1);
            \draw (0,5) rectangle +(1,1);
            \draw (2.5,0.5) node{$+$};
            \draw (3.5,0.5) node{$+$};
            \draw (0.5,1.5) node{$+$};
            \draw (2.5,1.5) node{$+$};
            \draw (0.5,2.5) node{$+$};
            \draw (0.5,3.5) node{$+$};
            \draw (0.5,4.5) node{$+$};
		\end{tikzpicture} \quad
        \begin{tikzpicture}[x=1em,y=-1em]
			\draw (0,0) rectangle +(1,1);
			\draw (1,0) rectangle +(1,1);
			\draw (2,0) rectangle +(1,1);
			\draw (3,0) rectangle +(1,1);
            \draw (4,0) rectangle +(1,1);
            \draw (5,0) rectangle +(1,1);
			\draw (0,1) rectangle +(1,1);
			\draw (1,1) rectangle +(1,1);
			\draw (2,1) rectangle +(1,1);
            \draw (3,1) rectangle +(1,1);
            \draw (4,1) rectangle +(1,1);
			\draw (0,2) rectangle +(1,1);
			\draw (1,2) rectangle +(1,1);
            \draw (2,2) rectangle +(1,1);
            \draw (3,2) rectangle +(1,1);
            \draw (0,3) rectangle +(1,1);
			\draw (1,3) rectangle +(1,1);
            \draw (2,3) rectangle +(1,1);
            \draw (0,4) rectangle +(1,1);
            \draw (1,4) rectangle +(1,1);
            \draw (0,5) rectangle +(1,1);
            \draw (2.5,0.5) node{$+$};
            \draw (0.5,1.5) node{$+$};
            \draw (1.5,1.5) node{$+$};
            \draw (2.5,1.5) node{$+$};
            \draw (0.5,2.5) node{$+$};
            \draw (0.5,3.5) node{$+$};
            \draw (0.5,4.5) node{$+$};
		\end{tikzpicture}
		\caption{}	
\label{f.mitosisA}	
	\end{figure}
If $i=2,3,4,6$, then $\MA_i(D)$ is empty.
Finally, $\MA_5(D)$ consists of a single offspring obtained from $D$ by erasing the $+$ in the fifth row.
\end{example}
\begin{remark} \label{r.A} It is easy to check that $\MA_i$ coincides with the $i$-th mitosis operator ${\rm mitosis}_i$ introduced in \cite[Definition 6]{M} (Knutson--Miller mitosis).
There is also a strong relationship between $\MA_i$ and the operator $M_i$ introduced in \cite[Section 5]{F} using representation theoretic considerations.
These operations coincide whenever $D$ does not have any $+$ in cells to the right of $r_{D_i}$.
While $\MA_i$ and $M_i$ might differ on the other pipe dreams (the latter operation might produce more offsprings than the former) it is interesting that both operations lead to the same results in Schubert calculus (see \cite[Corollary 5.13]{F} and the preceding discussion for more details).
In particular, \cite[Theorem 15]{M} still holds if ${\rm mitosis}_i$ is replaced by $M_i$, that is, $M_i$ can be viewed as an alternative version of the Knutson--Miller mitosis.
\end{remark}
\subsection{Mitosis in type $C$}
Similarly to type $A$ case, pipe dreams in type $C_n$ (also called {\em skew pipe dreams}) can be defined as $n\times(2n-1)$ tables filled with $+$.
In type $C_n$ case, $+$ is not allowed in cell $(i,j)$ if $i+j> 2n$ or $i>j$.
An example of a skew pipe dream for $n=4$ is given on Figure \ref{f.pipe_dream} (right).
Skew pipe dreams were recently used in \cite{F} to construct a convex geometric model for Schubert calculus in type $C$ in terms of symplectic GZ polytopes (we say more about this in Section \ref{s.Schubert}), in particular, they are related to signed permutation.
It would be interesting to find an interpretation of skew pipe dreams in terms of  networks of pipes.
In \cite{ST}, {\em c-signed pipe dreams} were defined  as networks of pipes with extra features.
We do not know of any direct relation between skew pipe dreams and c-signed pipe dreams.

We now define $n$ mitosis operations $\MC_1$,\ldots, $\MC_n$ on a skew pipe dream $D$.
The case $i=1$ is special.
In this case, set $\ell=n-1$.
Label cells $(1,n)$, $(2,n)$,\ldots, $(n,n)$ by $a1$, $a2$,\ldots, $a(\ell+1)$, respectively (these are cells in the middle column of $D$).
Label cells $(1,n+1)$, $(2,n+1)$,\ldots, $(n-1,n+1)$ by $b1$, $b2$,\ldots, $b\ell$, respectively.
If $i=2$,\ldots, $n$, set $\ell=2(n-i)+1$.
Label cells $(1,n-i+1)$, $(1,n+i-1)$, $(2,n-i+1)$, $(2,n+i+1)$,\ldots , $(n-i+1,n-i+1)$, $(n-i+1,n+i-1)$ by $a1$, $a2$, $a3$, $a4$,\ldots, $a\ell$, $a(\ell+1)$, respectively (these are all cells in columns $n\pm(i-1)$ of $D$).
Label cells $(1,n-i+2)$, $(1,n+i)$, $(2,n-i+2)$, $(2,n+i)$,\ldots , $(n-i+1,n-i+2)$  by $b1$, $b2$, $b3$, $b4$,\ldots, $b\ell$, respectively (here we alternate first $n-i+1$ cells in column $n-i+2$ with cells in column $n+i$ of $D$).
The rest of the definition of $\MC_i$ is completely analogous to type $A$ case.
\begin{example}
Let $D$ be a skew pipe dream from Figure \ref{f.pipe_dream} (right).
Using Example \ref{e.2mitosis} again we get that $\MC_1(D)$ consists of skew pipe dreams depicted on Figure \ref{f.mitosisC}.
\begin{figure}[h]
		\centering
    \begin{tikzpicture}[x=1em,y=-1em]
			\draw (0,0) rectangle +(1,1);
			\draw (1,0) rectangle +(1,1);
			\draw (2,0) rectangle +(1,1);
			\draw (3,0) rectangle +(1,1);
            \draw (4,0) rectangle +(1,1);
            \draw (5,0) rectangle +(1,1);
			\draw (6,0) rectangle +(1,1);
			\draw (1,1) rectangle +(1,1);
            \draw (2,1) rectangle +(1,1);
            \draw (3,1) rectangle +(1,1);
			\draw (4,1) rectangle +(1,1);
			\draw (5,1) rectangle +(1,1);
            \draw (2,2) rectangle +(1,1);
            \draw (3,2) rectangle +(1,1);
            \draw (4,2) rectangle +(1,1);
			\draw (3,3) rectangle +(1,1);
			\draw (3.5,1.5) node{$+$};
			\draw (3.5,2.5) node{$+$};
            \draw (3.5,3.5) node{$+$};
			\draw (4.5,2.5) node{$+$};
            \draw (2.5,1.5) node{$+$};
            \draw (2.5,2.5) node{$+$};
            \draw (2.5,0.5) node{$+$};
            \draw (0.5,0.5) node{$+$};
		\end{tikzpicture}\quad
        \begin{tikzpicture}[x=1em,y=-1em]
			\draw (0,0) rectangle +(1,1);
			\draw (1,0) rectangle +(1,1);
			\draw (2,0) rectangle +(1,1);
			\draw (3,0) rectangle +(1,1);
            \draw (4,0) rectangle +(1,1);
            \draw (5,0) rectangle +(1,1);
			\draw (6,0) rectangle +(1,1);
			\draw (1,1) rectangle +(1,1);
            \draw (2,1) rectangle +(1,1);
            \draw (3,1) rectangle +(1,1);
			\draw (4,1) rectangle +(1,1);
			\draw (5,1) rectangle +(1,1);
            \draw (2,2) rectangle +(1,1);
            \draw (3,2) rectangle +(1,1);
            \draw (4,2) rectangle +(1,1);
			\draw (3,3) rectangle +(1,1);
            \draw (4.5,0.5) node{$+$};
			\draw (3.5,2.5) node{$+$};
            \draw (3.5,3.5) node{$+$};
			\draw (4.5,2.5) node{$+$};
            \draw (2.5,1.5) node{$+$};
            \draw (2.5,2.5) node{$+$};
            \draw (2.5,0.5) node{$+$};
            \draw (0.5,0.5) node{$+$};
		\end{tikzpicture}\quad
        \begin{tikzpicture}[x=1em,y=-1em]
			\draw (0,0) rectangle +(1,1);
			\draw (1,0) rectangle +(1,1);
			\draw (2,0) rectangle +(1,1);
			\draw (3,0) rectangle +(1,1);
            \draw (4,0) rectangle +(1,1);
            \draw (5,0) rectangle +(1,1);
			\draw (6,0) rectangle +(1,1);
			\draw (1,1) rectangle +(1,1);
            \draw (2,1) rectangle +(1,1);
            \draw (3,1) rectangle +(1,1);
			\draw (4,1) rectangle +(1,1);
			\draw (5,1) rectangle +(1,1);
            \draw (2,2) rectangle +(1,1);
            \draw (3,2) rectangle +(1,1);
            \draw (4,2) rectangle +(1,1);
			\draw (3,3) rectangle +(1,1);
            \draw (4.5,0.5) node{$+$};
			\draw (4.5,1.5) node{$+$};
			\draw (3.5,2.5) node{$+$};
			\draw (4.5,2.5) node{$+$};
            \draw (2.5,1.5) node{$+$};
            \draw (2.5,2.5) node{$+$};
            \draw (2.5,0.5) node{$+$};
            \draw (0.5,0.5) node{$+$};
		\end{tikzpicture}
		\caption{}
		\label{f.mitosisC}
	\end{figure}
For $i=2,3$, the set $\MC_i(D)$ is empty.
The set $\MC_4(D)$ consists of a single offspring obtained by erasing the $+$ in the first column of $D$.
\end{example}

\begin{remark}\label{r.mitosisC}
Mitosis operation $\MC_i$ on skew pipe dreams is different from operator $M_i$ defined in \cite[Section 6]{F}.
As in type $A$ (see Remark \ref{r.A}), the difference lies in the restriction $j\le r_D$ in Definition \ref{d.mitosis}.
It would be interesting to check whether both operations still produce the same results in applications to Schubert calculus.
\end{remark}

\section{Mitosis on GZ polytopes in types A and C}\label{s.Schubert}
In this section, we apply simple geometric mitosis to GZ polytopes in types $A$ and $C$.
We compare the resulting operations with geometric realizations of Knutson--Miller mitosis and Fujita mitosis.

\subsection{Gelfand--Zetlin polytopes in type $A_n$}
Put $d:=\frac{n(n+1)}{2}$.
A {\em GZ table of type $A_n$} is a collection of cells organized according to the pattern on Figure \ref{f.patterns} (left).
Let $\l:=(\l_1\ge\l_2\ge\ldots\ge\l_{n+1})$ be a decreasing collection of real numbers.
Identify a point $(x^1_1,x^1_2,\ldots, x^1_n; x^2_1,\ldots,x^2_{n-1};\ldots; x^{n-1}_1, x^{n-1}_2; x^{n}_1)\in\R^d$ with the GZ table whose $i$-th row is filled with coordinates $(x^i_1,x^i_2,\ldots,x^i_{n-i+1})$ for $1\le i\le n$.
Define the GZ polytope $GZ^A_\l\subset\R^d$ by $2d$ interlacing inequalities $x^{i-1}_j\ge x^i_j\ge x^{i-1}_{j+1}$ (we put $x^0_j:=\l_j$).
In terms of the GZ table, these inequalities tell us that the coordinate in any given cell lies between coordinates in two upper neighbors of this cell.
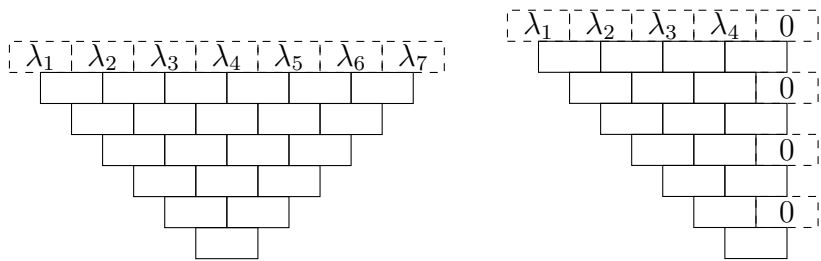
\begin{figure}[h]
		\centering
		\begin{tikzpicture}[x=2em,y=-1em]
			\draw[dashed] (-0.5,-1) rectangle +(1,1);
            \draw[dashed] (0.5,-1) rectangle +(1,1);
            \draw[dashed] (1.5,-1) rectangle +(1,1);
            \draw[dashed] (2.5,-1) rectangle +(1,1);
            \draw[dashed] (3.5,-1) rectangle +(1,1);
            \draw[dashed] (4.5,-1) rectangle +(1,1);
            \draw[dashed] (5.5,-1) rectangle +(1,1);
            \draw (0,0) rectangle +(1,1);
			\draw (1,0) rectangle +(1,1);
			\draw (2,0) rectangle +(1,1);
			\draw (3,0) rectangle +(1,1);
            \draw (4,0) rectangle +(1,1);
            \draw (5,0) rectangle +(1,1);
			\draw (0.5,1) rectangle +(1,1);
			\draw (1.5,1) rectangle +(1,1);
			\draw (2.5,1) rectangle +(1,1);
            \draw (3.5,1) rectangle +(1,1);
            \draw (4.5,1) rectangle +(1,1);
			\draw (1,2) rectangle +(1,1);
			\draw (2,2) rectangle +(1,1);
            \draw (3,2) rectangle +(1,1);
            \draw (4,2) rectangle +(1,1);
            \draw (1.5,3) rectangle +(1,1);
			\draw (2.5,3) rectangle +(1,1);
            \draw (3.5,3) rectangle +(1,1);
            \draw (2,4) rectangle +(1,1);
            \draw (3,4) rectangle +(1,1);
            \draw (2.5,5) rectangle +(1,1);
			\draw (0,-0.5) node{$\l_1$};
			\draw (1,-0.5) node{$\l_2$};
            \draw (2,-0.5) node{$\l_3$};
            \draw (3,-0.5) node{$\l_4$};
            \draw (4,-0.5) node{$\l_5$};
            \draw (5,-0.5) node{$\l_6$};
            \draw (6,-0.5) node{$\l_7$};
		\end{tikzpicture}\qquad
		\begin{tikzpicture}[x=2em,y=-1em]
            \draw[dashed] (-0.5,-1) rectangle +(1,1);
            \draw[dashed] (0.5,-1) rectangle +(1,1);
            \draw[dashed] (1.5,-1) rectangle +(1,1);
            \draw[dashed] (2.5,-1) rectangle +(1,1);
            \draw[dashed] (3.5,-1) rectangle +(1,1);
            \draw[dashed] (3.5,1) rectangle +(1,1);
            \draw[dashed] (3.5,3) rectangle +(1,1);
            \draw[dashed] (3.5,5) rectangle +(1,1);
			\draw (0,0) rectangle +(1,1);
			\draw (1,0) rectangle +(1,1);
			\draw (2,0) rectangle +(1,1);
			\draw (3,0) rectangle +(1,1);
			\draw (0.5,1) rectangle +(1,1);
			\draw (1.5,1) rectangle +(1,1);
			\draw (2.5,1) rectangle +(1,1);
			\draw (1,2) rectangle +(1,1);
			\draw (2,2) rectangle +(1,1);
            \draw (3,2) rectangle +(1,1);
			\draw (1.5,3) rectangle +(1,1);
            \draw (2.5,3) rectangle +(1,1);
            \draw (2,4) rectangle +(1,1);
            \draw (3,4) rectangle +(1,1);
            \draw (2.5,5) rectangle +(1,1);
            \draw (3,6) rectangle +(1,1);
			\draw (0,-0.5) node{$\l_1$};
			\draw (1,-0.5) node{$\l_2$};
			\draw (2,-0.5) node{$\l_3$};
			\draw (3,-0.5) node{$\l_4$};
            \draw (4,-0.5) node{$0$};
            \draw (4,1.5) node{$0$};
            \draw (4,3.5) node{$0$};
            \draw (4,5.5) node{$0$};
		\end{tikzpicture}
		\caption{Gelfand--Zetlin patterns in type $A_6$ (left) and type $C_4$ (right).}
		\label{f.patterns}
	\end{figure}

Let $v$ be the {\em Kogan vertex} of $GZ_\l^A$.
Recall that the {\em Kogan vertex} is a unique vertex of the polytope $GZ^A_\l$ that satisfies simultaneously all $d$ equations of type $x^{i-1}_j= x^i_j$.
A face $\G\subset GZ^A_\l$ is called a {\em Kogan face} if $\G$ contains $v$.
These faces were first considered by Mikhail Kogan \cite{Ko}.
In particular, $v$-faces used in the definition of mitosis in Section \ref{s.mitosis} are exactly Kogan faces for this choice of $v$.

Note that the polytope $GZ^A_\l$ for $\l=(n,n-1,\ldots,1,0)$ can be represented as the Cayley sum of two polytopes in $n$ differents ways.
Namely, let $P_i$ and $Q_i$ be the facets of $GZ^A_\l$ given by equations $x^1_i=\l_i$ and $x^1_i=\l_{i+1}$, respectively.
Then $\D(P_i,Q_i)$ coincides with $GZ^A_\l$ up to a parallel translation.
Clearly, $v\in P_1,$\ldots, $P_n$.
Hence, we can define $n$ different mitosis operations $\M^v_1$,\ldots, $\M^v_n$ on Kogan faces of $GZ_\l^A$ as in Section \ref{s.mitosis}.
In particular, Example \ref{e.FFLV} describes these operations for $n=2$ because the GZ polytope in this case is unimodularly equivalent to the FFLV polytope.

By definition, any Kogan face is given only by equations of type $x^{i-1}_j= x^i_j$.
Hence, every Kogan face $F$ can be encoded by a GZ table filled with $+$.
Namely, if $F$ lies in the hyperplane $\{x^{i-1}_j= x^i_j\}$ then put $+$ in cell $(i,j)$ (that is, in the $j$-th cell of the $i$-th row).
Otherwise, leave cell $(i,j)$ empty.
The resulting GZ table filled with $+$ will be called the {\em diagram} of the face $F$ and denoted by $D(F)$.

Define a bijective correspondence between pipe dreams in type $A_n$ and diagrams of Kogan faces of $GZ^A_\l$ by identifying cell $(i,j)$ in a GZ table with cell $(j,i)$ in a pipe dream.
The correspondence is illustrated by Figure \ref{f.bijection} (left).
Namely, inscribe the words ``GELFAND ZETLIN POLYTOPE'' (without spaces) inside a pipe dream of type $A_6$ in usual way, that is, start from the top row and write from left to right.
After switching to $GZ$ table, the words will transform to the sequence of letters on Figure \ref{f.bijection} (left).

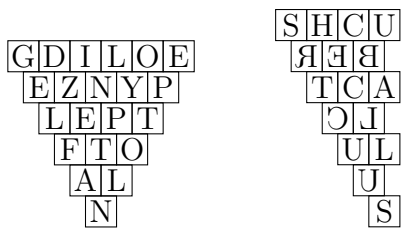
\begin{figure}[h]
		\centering
\begin{tikzpicture}[x=1em,y=-1em]
			\draw (0,0) rectangle +(1,1);
			\draw (1,0) rectangle +(1,1);
			\draw (2,0) rectangle +(1,1);
			\draw (3,0) rectangle +(1,1);
            \draw (4,0) rectangle +(1,1);
            \draw (5,0) rectangle +(1,1);
			\draw (0.5,1) rectangle +(1,1);
			\draw (1.5,1) rectangle +(1,1);
			\draw (2.5,1) rectangle +(1,1);
            \draw (3.5,1) rectangle +(1,1);
            \draw (4.5,1) rectangle +(1,1);
			\draw (1,2) rectangle +(1,1);
			\draw (2,2) rectangle +(1,1);
            \draw (3,2) rectangle +(1,1);
            \draw (4,2) rectangle +(1,1);
            \draw (1.5,3) rectangle +(1,1);
			\draw (2.5,3) rectangle +(1,1);
            \draw (3.5,3) rectangle +(1,1);
            \draw (2,4) rectangle +(1,1);
            \draw (3,4) rectangle +(1,1);
            \draw (2.5,5) rectangle +(1,1);
			\draw (0.5,0.5) node{G};
			\draw (1,1.5) node{E};
			\draw (1.5,2.5) node{L};
			\draw (2,3.5) node{F};
            \draw (2.5,4.5) node{A};
			\draw (3,5.5) node{N};
			\draw (1.5,0.5) node{D};
			\draw (2,1.5) node{Z};
            \draw (2.5,2.5) node{E};
            \draw (3,3.5) node{T};
            \draw (3.5,4.5) node{L};
            \draw (2.5,0.5) node{I};
            \draw (3,1.5) node{N};
            \draw (3.5,2.5) node{P};
            \draw (4,3.5) node{O};
            \draw (3.5,0.5) node{L};
            \draw (4,1.5) node{Y};
            \draw (4.5,2.5) node{T};
            \draw (4.5,0.5) node{O};
            \draw (5,1.5) node{P};
            \draw (5.5,0.5) node{E};
		\end{tikzpicture} \qquad
		\begin{tikzpicture}[x=1em,y=-1em]
			\draw (0,0) rectangle +(1,1);
			\draw (1,0) rectangle +(1,1);
			\draw (2,0) rectangle +(1,1);
			\draw (3,0) rectangle +(1,1);
			\draw (0.5,1) rectangle +(1,1);
			\draw (1.5,1) rectangle +(1,1);
			\draw (2.5,1) rectangle +(1,1);
			\draw (1,2) rectangle +(1,1);
			\draw (2,2) rectangle +(1,1);
            \draw (3,2) rectangle +(1,1);
			\draw (1.5,3) rectangle +(1,1);
            \draw (2.5,3) rectangle +(1,1);
            \draw (2,4) rectangle +(1,1);
            \draw (3,4) rectangle +(1,1);
            \draw (2.5,5) rectangle +(1,1);
            \draw (3,6) rectangle +(1,1);
			\draw (0.5,0.5) node{S};
			\draw (1.5,0.5) node{H};
			\draw (2.5,0.5) node{C};
			\draw (3.5,0.5) node{U};
            \draw (3,1.5) node{\reflectbox{B}};
			\draw (2,1.5) node{\reflectbox{E}};
			\draw (1,1.5) node{\reflectbox{R}};
			\draw (1.5,2.5) node{T};
            \draw (2.5,2.5) node{C};
            \draw (3.5,2.5) node{A};
            \draw (3,3.5) node{\reflectbox{L}};
            \draw (2,3.5) node{\reflectbox{C}};
            \draw (2.5,4.5) node{U};
            \draw (3.5,4.5) node{L};
            \draw (3,5.5) node{\reflectbox{U}};
            \draw (3.5,6.5) node{S};
		\end{tikzpicture}
		\caption{Correspondence between pipe dreams and GZ patterns in type $A_6$ (left) and type $C_4$ (right).}
		\label{f.bijection}
\end{figure}

\begin{example}
The pipe dream on Figure \ref{f.pipe_dream}  (left) corresponds to the Kogan face $F$ given by eight equations $\l_1=x^1_1=x^2_1=x^3_1=x^4_1$, $x^2_2=x^3_2$, $\l_3=x^1_3$, $\l_4=x^1_4$, $\l_5=x^1_5$.
In particular, $\dim F=13$.
It is easy to check that $\M_1^v(F)$ consists of three Kogan faces of dimension $14$ whose diagrams correspond to pipe dreams on Figure \ref{f.mitosisA}.
\end{example}

Using the bijection between pipe dreams and GZ tables we can extend mitosis operations $\MA_1$,\ldots, $\MA_n$ from pipe dreams to GZ tables.
Informally speaking, we replace rows with columns when applying two-row mitosis.
We will use the same notation for mitosis operations on GZ tables as it will be clear from the context whether we talk about pipe dreams or about GZ tables.

It turns out that geometric operation $\M_i^v$ on certain Kogan faces is combinatorially equivalent to mitosis operation $\MA_i$ on the corresponding pipe dreams.
The following theorem describes the precise relation between the geometric mitosis on Kogan faces of a GZ polytope and the combinatorial mitosis on GZ tables.

\begin{thm}\label{t.main}
Let $F\subset GZ_\l^A$ be a Kogan face such that
its diagram $D(F)$ has $+$ in cell $(1,i)$ and no $+$ in cells $(j,i+1)$ for all $j$.
Then $\M_i^v(F)$ consists of Kogan faces whose diagrams are obtained from $D(F)$ by the mitosis operation $\MA_i$ on GZ tables.
\end{thm}
\begin{proof}
In order to compute $\M_i^v(F)$ we have to consider not only Kogan faces.
In general, a face of $GZ^A_\l$ is given by equations of two types: either $x^{i-1}_{j}=x^i_{j}$ (type $A$) or $x^i_j=x^{i-1}_{j+1}$ (type $B$).
\begin{defin}\label{d.AB}
Following \cite{Ko} define equations $A_{i,j}$ and $B_{i,j}$, respectively, as
$x^{i-1}_{j}=x^i_{j}$ and $x^i_j=x^{i-1}_{j+1}$.
\end{defin}

We now apply Definition \ref{d.mitosis} to $P:=P_i$ and $Q:=Q_i$.
Note that any Kogan face is admissible because it contains a simple vertex, namely, the Kogan vertex $v$.
The face $E_1:=\exp(F)$ is obtained from $F$ by removing a single equation $A_{1,i}$.
Hence, the face $\exp(F)\cap Q$ is given by the same equations as $F$ with one exception: equation $A_{1,i}$ is replaced by equation $B_{1,i}$.

Let $E$ be a Kogan face of dimension $\ell+1$ (where $\ell:=\dim F$) such that $E\cap Q\subset \exp(F)\cap Q$.
Let ($X_1$,\ldots, $X_{d-\ell-1}$) be the collection of equations of type $A$ that define $E$.
If $E\ne \exp(F)$ then there exists an equation $Y$ (of type A) such that $\exp(F)$ satisfies $Y$ but $Y$ does not follow from ($X_1$,\ldots, $X_{d-\ell-1}$).
However, $Y$ should follow from equations $B_{1,i}$, $X_1$,\ldots, $X_{d-\ell-1}$.
Hence, the collection of equations ($B_{1,i}$, $Y$, $X_1$,\ldots, $X_{d-\ell-1}$) is redundant, while the collection ($Y$, $X_1$,\ldots, $X_{d-\ell-1}$) is not.
This is only possible if $A_{1,i+1}$ is contained among the equations $X_1$,\ldots, $X_{d-\ell-1}$ as the following lemma shows:
\begin{lemma} \label{l.GZ} Let $(X_1,\ldots, X_{s})$ be a collection of equations of type $A$ that does not contain equation $A_{p,q}$.
If $A_{p,q}$ follows from $(B_{j,i},X_1,\ldots, X_s)$ then three conditions hold:
\begin{enumerate}
  \item $q=i$;
  \item $p>j$;
  \item $(X_1,\ldots, X_s)$ contains $A_{k,i+1}$ for all $k$ such that $j\le k<p$.
\end{enumerate}
\end{lemma}

\begin{proof}
The statement follows directly from Definition \ref{d.AB}.
\end{proof}
If ($X_1$, $X_2$,\ldots, $X_{d-\ell-1}$) contains $A_{1,i+1}$, $A_{2,i+1}$,\ldots, $A_{k,i+1}$ but does not contain $A_{k+1,i+1}$ for some $k\ge1$,
then $A_{2,i}$,\ldots, $A_{k+1, i}$ are the only equations that follow from ($B_{1,i}$, $X_1$,\ldots, $X_{d-\ell-1})$.
Hence, if $F$ satisfies $A_{p,q}$, and $A_{p,q}$ is not contained among ($X_1$, $X_2$,\ldots, $X_{d-\ell-1}$)) then $A_{p,q}$ must coincide with one of the equations $A_{2,i}$,\ldots, $A_{k+1, i}$.
For dimension reasons, all $k$ equations $A_{2,i}$,\ldots, $A_{k+1, i}$ must be absent in ($X_1$, $X_2$,\ldots, $X_{d-\ell-1}$) (to compensate for the presence of $k$ equations $A_{1,i+1}$, $A_{2,i+1}$,\ldots, $A_{k,i+1}$).
In particular, if $F$ does not satisfy $A_{2,i}$ then $\M_i^v(F)$ consists of a single offspring $\exp(F)$.

We now proceed by induction on $r$ where $r$ is the maximal number such that $F$ satisfies the equations $A_{1,i}$,\ldots , $A_{r,i}$.
Let $\G_{r,i}$ be the facet given by equation $A_{r,i}$, and $F'$ the face given by the same equations as $F$ except for $A_{r,i}$, that is, $F=F'\cap \G_{r,i}$.
It is easy to check that there is a bijection between $\M^v_i(F')$ and
the subset $S\subset \M^v_i(F)$ that consists of all faces $E_i\in \M^v_i(F)$ such that $E_i \subset \G_{r,i}$.
Namely, a face $E'\in \M^v_i(F')$ corresponds to the face $E=E'\cap \G_{r,i}\in \M^v_i(F)$.
Hence, the faces in $S$ can be described by the induction hypothesis applied to $F'$.

It remains to describe the faces $E_i\in\M^v_i(F)\setminus S$.
By condition (2) of Definition \ref{d.mitosis} and part (3) of Lemma \ref{l.GZ} applied to $A_{r,i}$, we have that if $E_i\not\subset \G_{r,i}$, then $E_i$ must satisfy equations $A_{1,i+1}$, $A_{2,i+1}$\ldots, $A_{r-1,i+1}$.
Hence, there is a single offspring
$E_r\in\M^v_i(F)\setminus S$ obtained from $F$ by removing equation $A_{r,i}$ and by replacing equations $A_{j,i}$ by equations $A_{j,i+1}$ for all $j<r$.
\end{proof}

We now adapt geometric mitosis operations so that Theorem \ref{t.main} holds for all {\em reduced} Kogan faces.
The definition of {\em reduced} pipe dreams can be found in \cite[Definition 1]{M} (these pipe dreams are mostly used in applications to Schubert calculus).
A Kogan face $F$ is {\em reduced} if its diagram $D(F)$ is reduced.
Below we use notation of Definition \ref{d.AB} from the proof of Theorem \ref{t.main}.

\begin{cor}\label{c.mitosisA}
Let $F\subset GZ_\l^A$ be a reduced Kogan face.
Define an adapted geometric mitosis operation $\mM_{i}^v$ as follows.
\begin{enumerate}
\item Consider the face $\env(F)\subset GZ_\l^A$ given by equations $A_{k,i}$ and $A_{k,i+1}$ for all $k$ such that $F$ satisfies both $A_{k,i}$ and $A_{k,i+1}$.
\item Find the minimal index $s$ such that $\env(F)$ does not satisfy equation $A_{s,i}$.
     Consider the facets $P_i^F$, $Q_i^F\subset \env(F)$ given by equations $A_{s,i}$ and $B_{s,i}$, respectively.
\item Define $\mM_{i}^v(F)$ by applying Definition \ref{d.mitosis} to $\Delta=\env(F)$, $P=P_i^F$ and $Q=Q_i^F$.
\end{enumerate}
Then $\mM_{i}^v(F)$ consists of exactly those Kogan faces whose diagrams lie in the set $\MA_i(D(F))$.
\end{cor}
The proof of Corollary \ref{c.mitosisA} is completely analogous to the proof of Theorem \ref{t.main}.
\begin{remark}
Note that the set $\M_{i}^v(F)$ might, in general, contain more faces than the set
$\mM_{i}^v(F)$.
By applying mitosis inside the face $\env(F)$ instead of the whole $GZ_\l^A$ we get rid of these extra faces.
The definition of $\env(F)$ in Corollary \ref{c.mitosisA} is ad hoc, it relies heavily on combinatorics of GZ polytopes in type $A$.
It would be interesting to find a more geometric definition of $\env(F)$.

The choice of $P_i^F$ and $Q_i^F$ stems from the classical inductive construction of GZ bases and polytopes based on the chain of subgroups $GL_2(\C)\subset GL_3(\C)\subset\ldots\subset GL_{n+1}(\C)$.
\end{remark}

\subsection{Gelfand--Zetlin polytopes in type $C$}\label{s.C}
We now describe analogous constructions in type $C$.
We omit details that are the same as in type $A$ case and focus on unique features of type $C$ case.
Put $d=n^2$, and put $\l:=(\l_1\ge\l_2\ge\ldots\ge\l_n\ge \l_{n+1}=0)$.
A {\em GZ table of type $C_n$} is defined according to the pattern on Figure \ref{f.patterns} (right).
Roughly speaking, a GZ table of type $C_n$ is a half of a GZ table of type $A_{2n}$.
As in type $A$ case, the GZ polytope $GZ^C_\l\subset\R^d$ is defined by $2d$ interlacing inequalities that come from a $GZ$ pattern in type $C$.
Again, the polytope $GZ_\l^C$ for $\l=(n,n-1,\ldots,0)$ can be represented as Cayley sum of two polytopes in $n$ different ways.

Let $v$ be the {\em symplectic Kogan vertex} of $GZ_\l^C$ as defined in \cite[Section 6]{F}.
Using notation from Definition \ref{d.AB} (see the proof of Theorem \ref{t.main}) and regarding a GZ pattern of type $C_n$ as part of a GZ pattern of type $A_{2n-1}$ we may define $v$ by equations $A_{i,j}$ for all odd $i\le 2n-1$ and equations $B_{i,j}$ for all even $i<  2n-1$.
Similarly to type $A$ case, a face $\G\subset GZ^C_\l$ is called a {\em symplectic Kogan face} if $\G$ contains $v$.
Again, there are $n$ mitosis operations $\M_1^v$,\ldots, $\M_n^v$.

As in type $A$ case, every symplectic Kogan face $F$ can be encoded by a GZ table $D(F)$ of type $C$ filled with $+$.
Define bijective correspondence between skew pipe dreams and diagrams of Kogan faces as in \cite[Section 6]{F}.
For instance, if words ``SCHUBERT CALCULUS'' (without spaces) are inscribed into a skew pipe dream of type $C_4$ in usual way then they get transformed into boustrophedon\footnote{I am grateful to Evgeny Smirnov from whom I learnt the proper name of this writing style. } writing in a $GZ$ pattern of type $C_4$ on Figure \ref{f.bijection} (right).
Again, we extend mitosis operations $M_1^C$,\ldots, $M_n^C$ from skew pipe dreams to GZ tables in type $C_n$.

Note that $M_1^C$ is combinatorially different from $M_2^C$,\ldots, $M_n^C$.
In applications to Schubert calculus, mitosis operations correspond to generators of the group of signed permutations $B_n$ on $n$ elements.
There is a special generator $s_1:(1,2,\ldots,n)\mapsto (-1,2,\ldots,n)$ (change of sign), and the elementary transpositions $s_2$,\ldots,$s_n$, namely, $s_i=(i-1~i)$.
If we regard $B_n$ as the Weyl group of the symplectic group $Sp_{2n}(\C)$, then $s_1$ is the simple reflection corresponding to the longer root.
The mitosis operation $M_n^C$ corresponds to the special generator $s_1$, hence, it is natural to expect that $M_n^C$ will be special.

The following theorem is analogous to Theorem \ref{t.main} and can be proved using similar arguments.
\begin{thm}\label{t.C}
Let $F\subset GZ^C_\l$ be a symplectic Kogan face.
\begin{enumerate}
\item Let $i=1$,\ldots, $n-1$.
If the diagram $D(F)$ has $+$ in cell $(1,i)$, and no $+$ in cells $(2k+1,i-k+1)$ and $(2k,i-k)$ for all $k$, then $\M^v_i(F)$ consists of faces whose diagrams are obtained from $D(F)$ by applying $M^C_{n-i+1}$.
\item If the diagram $D(F)$ has $+$ in cell $(1,n)$, and no $+$ in cells $(2k,n-k)$ then $\M^v_n(F)$ consists of faces whose diagrams are obtained from $D(F)$ by applying $M^C_1$.
\end{enumerate}
\end{thm}
\section{Applications to Schubert calculus}\label{s.ppo}
We now explain how the notion of simple geometric mitosis fits into in the context of intersection theory.
In particular, we relate Corollary \ref{c.mitosisA} and Theorem \ref{t.C} with the analogous results on Schubert calculus in types $A$ and $C$.
Below we use the notion of {\em polytope ring} (aka {\em Khovanskii--Pukhlikov ring}).
The definition can be found, for instance, in \cite[Section 2]{K21}.
We refer the reader to \cite{F} for more details on applications of Khovanskii--Pukhlikov rings to Schubert calculus.

We use notation of Section \ref{s.mitosis}.
Let $R_\Delta$ and $R_P$ be the polytope rings of $\Delta$ and $P$.
Assume that $R_\Delta$ and $R_P$, respectively, are isomorphic to the Chow rings of smooth varieties $Y$ and $X$, where $Y=\P(E)$ is the projectivization of a rank two vector bundle $E$ on $X$ (see \cite[Section 4]{K21} for motivation behind such an assumption).
Let $p$ be the natural projection from $Y$ to $X$.
Then there is a push-pull operator
$$p^*p_*:CH^*(Y)\to CH^{*-1}(Y),$$
which is a homomorphism of $CH^*(X)$-modules.
The simple geometric mitosis is an attempt to describe explicitly the action of the push-pull operator on $R_\Delta\simeq CH^*(Y)$ using representations of elements of $R_\Delta$ by linear combinations of faces of $\Delta$.
It would be interesting to formalize the connection between Definition \ref{d.mitosis} and the action of $p^*p_*$ on faces of $\Delta$ in the general setting.
Below we will exhibit  such a connection in the special case of GZ polytopes and flag varieties.

Let $\Delta$ be the GZ polytope for a classical group $G$ and a strictly dominant weight $\l$.
Let $B\subset G$ denote a Borel subgroup of $G$, and let $Y:=G/B$ denote the complete flag variety for $G$.
The ring $R_\Delta$ is isomorphic to the subring of $CH^*(Y)$ generated by the first Chern classes of line bundles on $G/B$ \cite[Corollary 5.3, Remark 2.4]{Ka}.
For $G=GL_{n+1}(\C)$ and $Sp_{2n}(\C)$, this subring coincides with $CH^*(Y)$ (in general, the discrepancy between the subring and the whole $CH^*(Y)$ is measured by the {\em torsion index} of $G$, see \cite{T} for more details).
In particular, if $\l=\rho:=(n,n-1,\ldots,1,0)$, then $R_\Delta\simeq CH^*(Y)$ in types $A$ and $C$.

Denote by $\alpha_1$,\ldots, $\alpha_n$ the simple roots of $G$.
Let $X_i$ be the partial flag variety $G/P_i$ for the minimal parabolic subgroup $P_i$ corresponding to $\a_i$.
Note that the natural projection $p_i:Y\to X_i$ turns $Y$ into a $\P^1$-fibration (it can be realized as the projectivization of the rank two bundle $E_i:={p_i}_*\mathcal L(\rho)$ where $\mathcal L(\rho)$ is the line bundle on $Y$ corresponding to the weight $\rho$).
The classical divided difference (or push-pull) operator in Schubert calculus is defined as $\d_i:=p_i^*{p_i}_*$.

Recall that the operators $\d_1$,\ldots, $\d_n$ are used to generate the Schubert classes $[X_w]\in CH^*(Y)$ starting from the class $[X_{\rm id}]$ of a point.
More precisely, let $w\in W$ be an element of the Weyl group of $G$.
Choose a reduced decomposition $w=s_{i_1}\cdots s_{i_\ell}$.
Here $s_i$ denotes the reflection with respect to the root $\a_i$.
Then $[X_w]=\d_{i_\ell}\cdots\d_{i_1}[X_{\rm id}]$.
If we choose the Kogan vertex $v$ as a representative of the class $[X_{\rm id}]$ in $R_\Delta$, then the action of $\d_{i_\ell}\cdots\d_{i_1}$ on $[X_{\rm id}]$ can be computed using mitosis operations as follows.

\begin{thm}\label{t.SchubertA}
Let $G=GL_{n+1}(\C)$ and $\Delta=GZ^A_\rho$.
Under the isomorphism $CH^*(Y)\simeq R_\Delta$ the Schubert cycle $[X_w]$ can be represented as the class of the sum of faces $F\subset \Delta$ where $F$ runs through the set
$$S_w=\mM_{i_\ell}^v\cdots \mM^v_{i_1}(v).$$
In particular, the action of divided difference operator $\d_i$ on
the Schubert cycle $[X_w]$ gets represented by the action
of $\mM_{i}^v$ on faces from $S_w$:
$$S_{ws_i}=\mM_{i}^v(S_w).$$
\end{thm}
This theorem follows from \cite[Corollary 5.18]{F} together with Corollary \ref{c.mitosisA}.
Note that the proof of \cite[Theorem 5.17]{F} (which implies \cite[Corollary 5.18]{F}) uses representation theoretic arguments.
It would be interesting to find a convex geometric proof.
\begin{remark}
In \cite{K16}, different geometric mitosis operations are defined.
They mimick Demazure operators rather than push-pull operators $\d_i$.
In type $A$, they can also be used to prove theorems analogous to Theorem \ref{t.SchubertA} or \cite[Corollary 5.18]{F} due to a special symmetry of GZ diagrams in type $A$.
Namely, the diagrams are symmetric with respect to the reflection $(i,j)\mapsto (j,i)$.

Note that the induction step in \cite[Corollary 3.6]{K16} goes from $s_iw$ to $w$ (not from $ws_i$ to $w$ as in Theorem \ref{t.SchubertA} and in \cite[Theorem 5.17]{F}).
In other words, induction goes along initial subwords of $s_{i_1}s_{i_2}\ldots s_{i_{\ell}}$ (not along terminal subwords).
This difference is matched by the difference between mitosis as defined in \cite{K16} and its transpose (or mirror) mitosis as defined in the present paper.
In particular, all arguments with the mitosis operations applied to $w=s_{i_1}s_{i_2}\ldots s_{i_{\ell}}$ can be immediately translated into analogous arguments with the transpose mitosis operations applied to $w^{-1}=s_{i_\ell}s_{i_{\ell-1}}\ldots s_{i_1}$ (see \cite[Section 5]{F} for more details).
\end{remark}

In type $C$, Theorem \ref{t.C} together with Remark \ref{r.mitosisC} and \cite[Theorem 6.8]{F} suggest that $\M^v_i$ can be adapted so that under the isomorphism $CH^*(Y)\simeq R_\Delta$ the Schubert class $[X_w]$ can be represented as the class of the sum of faces $F\subset \Delta$ where $F$ runs through the set
$$S_w=\mM_{n-i_\ell+1}^v\cdots \mM^v_{n-i_1+1}(v).$$

In \cite[Corollary 6.13]{F}, another  presentation of Schubert cycles by faces of $GZ^C_\rho$ is obtained using the {\em dual Kogan faces}.
Similarly to the definition of symplectic Kogan vertex in Section \ref{s.C},
the {\em dual Kogan vertex} $v^*$ can be defined by equations $B_{i,j}$ for all odd $i\le 2n-1$ and equations $A_{i,j}$ for all even $i<  2n-1$.
Again there are $n$ geometric mitosis operations $\M^{v^*}_1$,\ldots, $\M^{v^*}_n$ where $\M^{v^*}_i$ corresponds to the decomposition $GZ^C_\rho=\{x^1_i=\l_{i+1}\}\star \{x^1_i=\l_i\}$ (that is, $P_i$ and $Q_i$ from the definition of $\M^{v}_i$ switch places when defining $\M^{v^*}_i$).
In the case of $Sp_4(\C)$, these operations produce part of the presentation of Schubert cycles from \cite[Corollary 6.13]{F} as the following example shows.

\begin{example}\label{e.dualC} Let $n=2$.
We encode dual symplectic Kogan faces by their diagrams.
We put $A$ and $B$ instead of $+$ just for clarity (not because it carries any extra information).
This way it is easier to distinguish diagrams of Kogan faces from diagrams of dual Kogan faces.
In the second row, we use an adapted mitosis once (this adaptation is based on the chain $Sp_2(\C)=SL_2(\C)\subset Sp_4(\C)$).
\begin{figure}[h]
\centering
		\begin{tikzpicture}[x=1em,y=-1em]
            \draw (2,4) rectangle +(1,1);
            \draw (3,4) rectangle +(1,1);
            \draw (2.5,5) rectangle +(1,1);
            \draw (3,6) rectangle +(1,1);
            \draw (2.5,4.5) node{B};
            \draw (3.5,4.5) node{B};
            \draw (6,5) node{$\stackrel{\M^{v^*}_1}{\longrightarrow}$};
            \draw (3,7.5) node{${\rm id}$};
            \draw (3,5.5) node{A};
            \draw (3.5,6.5) node{B};
		\end{tikzpicture}
        \begin{tikzpicture}[x=1em,y=-1em]
            \draw (10,4) rectangle +(1,1);
            \draw (11,4) rectangle +(1,1);
            \draw (10.5,5) rectangle +(1,1);
            \draw (11,6) rectangle +(1,1);
            \draw (10.5,4.5) node{};
            \draw (11.5,4.5) node{B};
            \draw (14,5) node{$\stackrel{\M^{v^*}_2}{\longrightarrow}$};
            \draw (11,7.5) node{$s_2$};
            \draw (11,5.5) node{A};
            \draw (11.5,6.5) node{B};
		\end{tikzpicture}
        \begin{tikzpicture}[x=1em,y=-1em]
            \draw (18,4) rectangle +(1,1);
            \draw (19,4) rectangle +(1,1);
            \draw (18.5,5) rectangle +(1,1);
            \draw (19,6) rectangle +(1,1);
            \draw (18.5,4.5) node{B};
            \draw (19.5,4.5) node{};
            \draw (22,5) node{\&};
            \draw (19,7.5) node{$s_2s_1$};
            \draw (19,5.5) node{};
            \draw (19.5,6.5) node{B};
		\end{tikzpicture}
        \begin{tikzpicture}[x=1em,y=-1em]
            \draw (22,4) rectangle +(1,1);
            \draw (23,4) rectangle +(1,1);
            \draw (22.5,5) rectangle +(1,1);
            \draw (23,6) rectangle +(1,1);
            \draw (22.5,4.5) node{};
            \draw (23.5,4.5) node{};
            \draw (23,7.5) node{$s_2s_1$};
            \draw (26,5) node{$\stackrel{\M^{v^*}_1}{\longrightarrow}$};
            \draw (23,5.5) node{A};
            \draw (23.5,6.5) node{B};
		\end{tikzpicture}
        \begin{tikzpicture}[x=1em,y=-1em]
            \draw (30,4) rectangle +(1,1);
            \draw (31,4) rectangle +(1,1);
            \draw (30.5,5) rectangle +(1,1);
            \draw (31,6) rectangle +(1,1);
            \draw (30.5,4.5) node{};
            \draw (31.5,4.5) node{B};
            \draw (34,5) node{\&};
            \draw (31,7.5) node{$s_2s_1s_2$};
            \draw (31,5.5) node{};
            \draw (31.5,6.5) node{};
		\end{tikzpicture}
        \begin{tikzpicture}[x=1em,y=-1em]
            \draw (34,4) rectangle +(1,1);
            \draw (35,4) rectangle +(1,1);
            \draw (34.5,5) rectangle +(1,1);
            \draw (35,6) rectangle +(1,1);
            \draw (34.5,4.5) node{};
            \draw (35.5,4.5) node{};
            \draw (34,5) node{};
            \draw (35,5.5) node{};
            \draw (35,7.5) node{$s_2s_1s_2$};
            \draw (35.5,6.5) node{B};
		\end{tikzpicture}

\begin{tikzpicture}[x=1em,y=-1em]
            \draw (2,4) rectangle +(1,1);
            \draw (3,4) rectangle +(1,1);
            \draw (2.5,5) rectangle +(1,1);
            \draw (3,6) rectangle +(1,1);
            \draw (2.5,4.5) node{B};
            \draw (3.5,4.5) node{B};
            \draw (6,5) node{$\stackrel{\mM^{v^*}_2}{\longrightarrow}$};
            \draw (3,7.5) node{$\rm id$};
            \draw (3,5.5) node{A};
            \draw (3.5,6.5) node{B};
		\end{tikzpicture}
        \begin{tikzpicture}[x=1em,y=-1em]
            \draw (10,4) rectangle +(1,1);
            \draw (11,4) rectangle +(1,1);
            \draw (10.5,5) rectangle +(1,1);
            \draw (11,6) rectangle +(1,1);
            \draw (10.5,4.5) node{B};
            \draw (11.5,4.5) node{B};
            \draw (14,5) node{$\stackrel{\M^{v^*}_1}{\longrightarrow}$};
            \draw (11,7.5) node{$s_1$};
            \draw (11,5.5) node{A};
            \draw (11.5,6.5) node{};
		\end{tikzpicture}
        \begin{tikzpicture}[x=1em,y=-1em]
            \draw (18,4) rectangle +(1,1);
            \draw (19,4) rectangle +(1,1);
            \draw (18.5,5) rectangle +(1,1);
            \draw (19,6) rectangle +(1,1);
            \draw (18.5,4.5) node{};
            \draw (19.5,4.5) node{B};
            \draw (22,5) node{$\stackrel{\M^{v^*}_2}{\longrightarrow}$};
            \draw (19,7.5) node{$s_1s_2$};
            \draw (19,5.5) node{A};
            \draw (19.5,6.5) node{};
		\end{tikzpicture}
        \begin{tikzpicture}[x=1em,y=-1em]
            \draw (26,4) rectangle +(1,1);
            \draw (27,4) rectangle +(1,1);
            \draw (26.5,5) rectangle +(1,1);
            \draw (27,6) rectangle +(1,1);
            \draw (26.5,4.5) node{B};
            \draw (27.5,4.5) node{};
            \draw (30,5) node{\&};
            \draw (27,7.5) node{$s_1s_2s_1$};
            \draw (27,5.5) node{};
            \draw (27.5,6.5) node{};
		\end{tikzpicture}
        \begin{tikzpicture}[x=1em,y=-1em]
            \draw (32,4) rectangle +(1,1);
            \draw (33,4) rectangle +(1,1);
            \draw (32.5,5) rectangle +(1,1);
            \draw (33,6) rectangle +(1,1);
            \draw (32.5,4.5) node{};
            \draw (33.5,4.5) node{};
            \draw (32,5) node{};
            \draw (33,7.5) node{$s_1s_2s_1$};
            \draw (33,5.5) node{A};
            \draw (33.5,6.5) node{};
		\end{tikzpicture}
		\caption{Mitosis for the dual Kogan vertex in type $C_2$.}
		\label{f.bijection}
\end{figure}
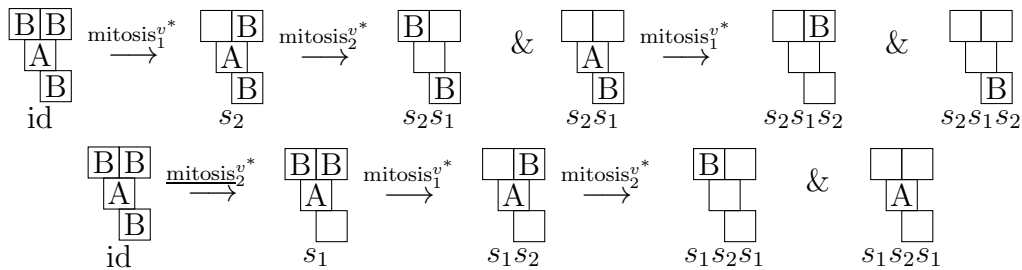
These collections of faces are the same as the presentations of Schubert cycles $[X_w]$ in \cite[Corollary 6.13]{F} for all $w\in W$ except for $s_2s_1$.
In the case $w=s_2s_1$, one face is missing.

Recall that according to \cite[Corollary 6.13]{F}, the dual Kogan faces that represent $[X_w]$ are in bijective correspondence with those reduced subwords of the longest word
$$\overline{w_0}=(s_1)(s_2s_1s_2)(s_3s_2s_1s_2s_3)\ldots(s_ns_{n-1}\ldots s_2s_1s_2\ldots s_{n-1}s_n)$$
that represent the element $w_0w$.
The bijection is obtained by inscribing the word $\overline{w_0}$ into the symplectic GZ pattern using reverse boustrophedon (see Figure \ref{f.dualKoganC} left).
For instance, the subword $$(s_1)(\cancel{s_2}\cancel{s_1}s_2)(s_3s_2s_1s_2\cancel{s_3})(\cancel{s_4}s_3
\cancel{s_2}\cancel{s_1}s_2s_3s_4)$$ in type $C_4$ corresponds to the dual Kogan face on Figure \ref{f.dualKoganC} right.

In particular, there are three dual Kogan faces that represent $[X_w]$ for $w=s_2s_1$ in type $C_2$.
Indeed, $w_0w=s_1s_2$, and there are three reduced subwords of $s_1(s_2s_1s_2)$ that represent $s_1s_2$.
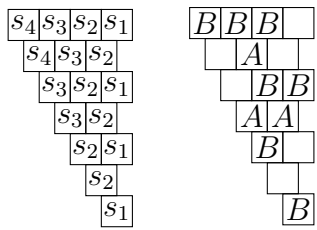
\begin{figure}[h]
		\centering
        \begin{tikzpicture}[x=1em,y=-1em]
			\draw (0,0) rectangle +(1,1);
			\draw (1,0) rectangle +(1,1);
			\draw (2,0) rectangle +(1,1);
			\draw (3,0) rectangle +(1,1);
			\draw (0.5,1) rectangle +(1,1);
			\draw (1.5,1) rectangle +(1,1);
			\draw (2.5,1) rectangle +(1,1);
			\draw (1,2) rectangle +(1,1);
			\draw (2,2) rectangle +(1,1);
            \draw (3,2) rectangle +(1,1);
			\draw (1.5,3) rectangle +(1,1);
            \draw (2.5,3) rectangle +(1,1);
           \draw (2,4) rectangle +(1,1);
            \draw (3,4) rectangle +(1,1);
            \draw (2.5,5) rectangle +(1,1);
            \draw (3,6) rectangle +(1,1);
			\draw (0.5,0.5) node{$s_4$};
			\draw (1.5,0.5) node{$s_3$};
			\draw (2.5,0.5) node{$s_2$};
			\draw (3.5,0.5) node{$s_1$};
            \draw (3,1.5) node{$s_2$};
			\draw (2,1.5) node{$s_3$};
			\draw (1,1.5) node{$s_4$};
			\draw (1.5,2.5) node{$s_3$};
         \draw (2.5,2.5) node{$s_2$};
          \draw (3.5,2.5) node{$s_1$};
           \draw (3,3.5) node{$s_2$};
           \draw (2,3.5) node{$s_3$};
            \draw (2.5,4.5) node{$s_2$};
            \draw (3.5,4.5) node{$s_1$};
            \draw (3,5.5) node{$s_2$};
           \draw (3.5,6.5) node{$s_1$};
		\end{tikzpicture}
\quad\begin{tikzpicture}[x=1em,y=-1em]
			\draw (0,0) rectangle +(1,1);
			\draw (1,0) rectangle +(1,1);
			\draw (2,0) rectangle +(1,1);
			\draw (3,0) rectangle +(1,1);
			\draw (0.5,1) rectangle +(1,1);
			\draw (1.5,1) rectangle +(1,1);
			\draw (2.5,1) rectangle +(1,1);
			\draw (1,2) rectangle +(1,1);
			\draw (2,2) rectangle +(1,1);
            \draw (3,2) rectangle +(1,1);
			\draw (1.5,3) rectangle +(1,1);
            \draw (2.5,3) rectangle +(1,1);
           \draw (2,4) rectangle +(1,1);
            \draw (3,4) rectangle +(1,1);
            \draw (2.5,5) rectangle +(1,1);
            \draw (3,6) rectangle +(1,1);
			\draw (0.5,0.5) node{$B$};
			\draw (1.5,0.5) node{$B$};
			\draw (2.5,0.5) node{$B$};
			\draw (3.5,0.5) node{};
            \draw (3,1.5) node{};
			\draw (2,1.5) node{$A$};
			\draw (1,1.5) node{};
			\draw (1.5,2.5) node{};
         \draw (2.5,2.5) node{$B$};
          \draw (3.5,2.5) node{$B$};
           \draw (3,3.5) node{$A$};
           \draw (2,3.5) node{$A$};
            \draw (2.5,4.5) node{$B$};
            \draw (3.5,4.5) node{};
            \draw (3,5.5) node{};
           \draw (3.5,6.5) node{$B$};
		\end{tikzpicture}
	\caption{Correspondence between subwords of $\overline{w_0}$ and dual Kogan faces in type $C_4$}
		\label{f.dualKoganC}
	\end{figure}
\end{example}
\begin{remark}
In type $A$, there is no combinatorial difference between presentation of Schubert cycles by Kogan faces and presentation by dual Kogan faces (see \cite[Theorem 5.25]{F}).
Both presentations can be related using the automorphism of the Dynkin diagram of type $A_n$ (on the level of GZ diagrams this automorphism corresponds to the reflection $(i,j)\mapsto (i, n+2-i-j)$).
In type $C$, there is a combinatorial difference already for $n=2$ (see \cite[Section 1]{F} for more details).
In type $A$, one can immediately recover mitosis on dual Kogan faces from  mitosis on Kogan faces.
Both operations will be combinatorially equivalent to the Knutson--Miller mitosis.
In type $C$, the Fujita mitosis on Kogan faces does not yield a combinatorially equivalent mitosis on dual Kogan faces.
\end{remark}
Example \ref{e.dualC} shows that Definition \ref{d.mitosis} is too simple to capture completely the action of $p^*p_*$ on faces of $\Delta$ in the general setting.
However, it may be regarded as the first approximation of this action.
It can also be used to make an educated guess about a possible combinatorial mitosis on dual Kogan faces in type $C$ and mitosis in type $D$.

\end{document}